\documentclass[10pt]{article}
\usepackage[margin=0.85in]{geometry}
\usepackage{amsmath,amssymb,amsthm}
\usepackage{booktabs}
\usepackage{float}
\usepackage{microtype}
\usepackage{tikz}
\usepackage[hidelinks]{hyperref}
\hypersetup{
  pdftitle={The Truncated Octahedral Graph Has Bondage Number Five},
  pdfauthor={Prateek R. Srivastava}
}

\newtheorem{theorem}{Theorem}
\newtheorem{lemma}{Lemma}
\newcommand{\gammav}{\gamma}
\newcommand{\Deltaa}{\Delta}
\title{The Truncated Octahedral Graph Has Bondage Number Five}
\author{Prateek R. Srivastava\thanks{ORCID:
\href{https://orcid.org/0000-0002-5046-0465}{0000-0002-5046-0465}.}\\
\small Independent researcher\\[-2pt]
\small \texttt{prs7786@g.rit.edu}}
\date{August 18, 2026}

\begin{document}
\maketitle

\begin{abstract}
For a graph $G$, its bondage number $b(G)$ is the minimum number of edges whose
deletion increases its domination number.  Dunbar, Haynes, Teschner, and
Volkmann conjectured in 1998 that every nontrivial planar graph satisfies
$b(G)\leq \Deltaa(G)+1$.  We show that the truncated octahedral graph $T$ has
$\gammav(T)=8$ and $b(T)=5$.  Since $T$ is planar and cubic, this gives
$b(T)=5>4=\Deltaa(T)+1$ and disproves the conjecture.  The finite parts of the
verification are exhaustive: the direct verifier checks candidate dominating
sets of sizes six, seven, and eight and all
$\binom{36}{4}=58{,}905$ four-edge sets.  The targeted discovery search and an
independently written verifier are described, and the complete C++20 verifier
is included in the source archive.
\end{abstract}

\section{Introduction}

A set $D\subseteq V(G)$ is \emph{dominating} if every vertex outside $D$ has a
neighbor in $D$.  The minimum cardinality of such a set is the domination
number $\gammav(G)$.  The \emph{bondage number}, introduced by Fink et
al.~\cite{Fink1990}, is
\[
 b(G)=\min\{|F|:F\subseteq E(G),\ \gammav(G-F)>\gammav(G)\}.
\]
In 1998 Dunbar, Haynes, Teschner, and Volkmann made the following conjecture
\cite{Dunbar1998}.

\medskip
\noindent\textbf{Planar bondage conjecture.}
Every nontrivial connected planar graph $G$ satisfies
$b(G)\leq\Deltaa(G)+1$.
\medskip

Kang and Yuan proved the near-sharp bound
$b(G)\leq\min\{8,\Deltaa(G)+2\}$ for planar graphs~\cite{Kang2000}.
Subsequent work verified the conjecture for several planar classes
\cite{Fischermann2003}; see Xu's survey~\cite{Xu2013}.  Thus a cubic
counterexample at the boundary of the Kang--Yuan bound must have bondage
number five.  This observation reduces the search to a sharply defined
extremal question.  The example below is the skeleton of an Archimedean solid.

\section{Targeted discovery strategy}\label{sec:strategy}

The search was organized around two tight inequalities, not around an
unrestricted census of graphs.  First, the Kang--Yuan theorem shows that a
cubic planar counterexample can only occur at $b(G)=5$.  Second, Hartnell and
Rall proved that for every edge $uv$,
\[
b(G)\leq d(u)+d(v)-1-|N(u)\cap N(v)|. \tag{1}
\]
For a triangle-free cubic graph the right side is also five
\cite{Hartnell1994}.  The target class was therefore chosen to be
triangle-free cubic polyhedral graphs that attain equality in (1).

Such graphs are the planar duals of simple triangulations of minimum degree
at least four.  They were generated without isomorphic duplication using
\textsc{plantri}~\cite{Brinkmann2007}.  For each graph, all minimum dominating
sets were encoded as bit masks.  If $D$ is a minimum dominating set and
$x\notin D$, define
\[
B(D,x)=\{xy\in E(G):y\in D\}.
\]
An edge set $F$ destroys $D$ precisely when $B(D,x)\subseteq F$ for some
$x\notin D$.  Hence the question whether at most four edges destroy every
minimum dominating set becomes a small exact hitting problem over these
bundles.  A branch search always selected an unhit dominating set with the
fewest feasible bundles.  This is the principal contraction of the search
space: it avoids enumerating every edge subset separately for every generated
graph.

The isomorphism-free census through order $24$ is summarized below.  The
``survivors'' row counts graphs for which the bundle search proved that no
set of at most four edges destroys all minimum dominating sets.
\begin{center}
\begin{tabular}{c@{\qquad}rrrrrrrrr}
\toprule
$|V(G)|$ & 8&10&12&14&16&18&20&22&24\\
\midrule
graphs tested &1&1&2&5&12&34&130&525&2472\\
survivors &0&0&0&0&0&0&0&0&1\\
\bottomrule
\end{tabular}
\end{center}
Canonical labeling with \textsc{nauty}~\cite{McKay2014} identified the unique
survivor as the truncated octahedral graph.  The census guided discovery only;
the theorem below is established from the explicit graph by a separate direct
verification.

\section{The graph and its domination number}

Let $T$ be the graph in Figure~\ref{fig:T}.  Its $24$ vertices and $36$ edges
form the skeleton of the truncated octahedron.  The displayed crossing-free
embedding has six quadrilateral and eight hexagonal faces, so $T$ is planar;
every vertex has degree three.  The figure was generated directly from the
edge list by a Tutte barycentric embedding: the hexagon
$0,1,5,11,6,2$ was fixed as a convex outer face and the remaining coordinates
were solved over the rational numbers.  Exact orientation predicates check
that all vertices are distinct, no vertex lies on a nonincident edge, and no
two nonincident edges cross.  The checks are repeated after rounding to the
six-decimal coordinates passed to TikZ; the induced rotation system also
confirms six four-faces and eight six-faces.  The labels in the figure are
used throughout.
For an entirely textual specification, its edge set is
\begin{center}\small
$\begin{aligned}
E(T)=\{&(0,1),(0,2),(0,3),(1,4),(1,5),(2,6),(2,7),(3,4),(3,8),\\
&(4,9),(5,10),(5,11),(6,11),(6,12),(7,12),(7,13),(8,13),(8,14),\\
&(9,15),(9,16),(10,16),(10,17),(11,17),(12,18),(13,19),(14,15),(14,19),\\
&(15,20),(16,20),(17,21),(18,21),(18,22),(19,22),(20,23),(21,23),(22,23)\}.
\end{aligned}$
\end{center}

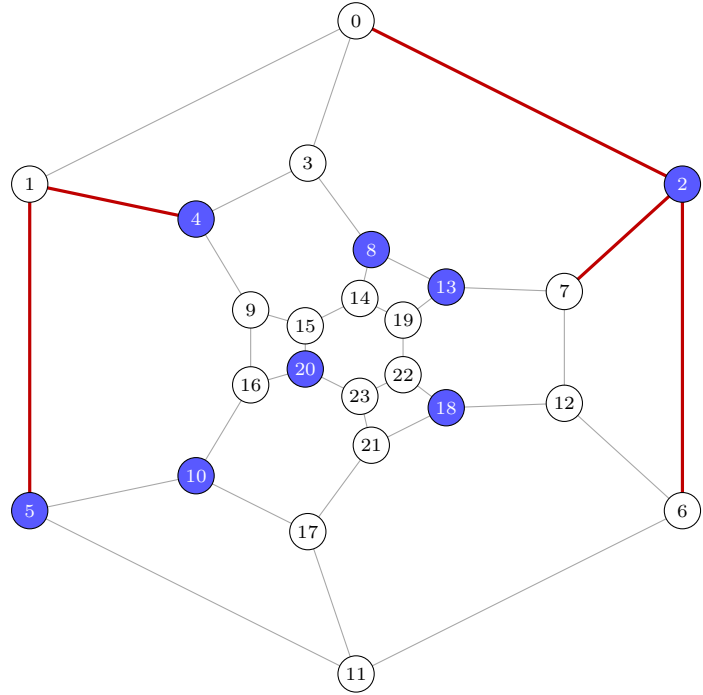
\begin{figure}[H]
\centering
\begin{tikzpicture}[scale=0.72,
  vertex/.style={circle,draw=black,fill=white,inner sep=0pt,minimum size=4.8mm,font=\scriptsize},
  dom/.style={vertex,fill=blue!65,text=white},
  every path/.style={line cap=round}]
\coordinate (v0) at (0.000000,6.000000);
\coordinate (v1) at (-6.000000,3.000000);
\coordinate (v2) at (6.000000,3.000000);
\coordinate (v3) at (-0.887160,3.385214);
\coordinate (v4) at (-2.941634,2.357977);
\coordinate (v5) at (-6.000000,-3.000000);
\coordinate (v6) at (6.000000,-3.000000);
\coordinate (v7) at (3.828794,1.027237);
\coordinate (v8) at (0.280156,1.797665);
\coordinate (v9) at (-1.937743,0.688716);
\coordinate (v10) at (-2.941634,-2.357977);
\coordinate (v11) at (0.000000,-6.000000);
\coordinate (v12) at (3.828794,-1.027237);
\coordinate (v13) at (1.657588,1.108949);
\coordinate (v14) at (0.070039,0.898833);
\coordinate (v15) at (-0.933852,0.396887);
\coordinate (v16) at (-1.937743,-0.688716);
\coordinate (v17) at (-0.887160,-3.385214);
\coordinate (v18) at (1.657588,-1.108949);
\coordinate (v19) at (0.863813,0.501946);
\coordinate (v20) at (-0.933852,-0.396887);
\coordinate (v21) at (0.280156,-1.797665);
\coordinate (v22) at (0.863813,-0.501946);
\coordinate (v23) at (0.070039,-0.898833);
\draw[gray!70] (v0)--(v1);
\draw[gray!70] (v0)--(v2);
\draw[gray!70] (v0)--(v3);
\draw[gray!70] (v1)--(v4);
\draw[gray!70] (v1)--(v5);
\draw[gray!70] (v2)--(v6);
\draw[gray!70] (v2)--(v7);
\draw[gray!70] (v3)--(v4);
\draw[gray!70] (v3)--(v8);
\draw[gray!70] (v4)--(v9);
\draw[gray!70] (v5)--(v10);
\draw[gray!70] (v5)--(v11);
\draw[gray!70] (v6)--(v11);
\draw[gray!70] (v6)--(v12);
\draw[gray!70] (v7)--(v12);
\draw[gray!70] (v7)--(v13);
\draw[gray!70] (v8)--(v13);
\draw[gray!70] (v8)--(v14);
\draw[gray!70] (v9)--(v15);
\draw[gray!70] (v9)--(v16);
\draw[gray!70] (v10)--(v16);
\draw[gray!70] (v10)--(v17);
\draw[gray!70] (v11)--(v17);
\draw[gray!70] (v12)--(v18);
\draw[gray!70] (v13)--(v19);
\draw[gray!70] (v14)--(v15);
\draw[gray!70] (v14)--(v19);
\draw[gray!70] (v15)--(v20);
\draw[gray!70] (v16)--(v20);
\draw[gray!70] (v17)--(v21);
\draw[gray!70] (v18)--(v21);
\draw[gray!70] (v18)--(v22);
\draw[gray!70] (v19)--(v22);
\draw[gray!70] (v20)--(v23);
\draw[gray!70] (v21)--(v23);
\draw[gray!70] (v22)--(v23);
\draw[red!75!black,very thick] (v0)--(v2);
\draw[red!75!black,very thick] (v1)--(v4);
\draw[red!75!black,very thick] (v1)--(v5);
\draw[red!75!black,very thick] (v2)--(v6);
\draw[red!75!black,very thick] (v2)--(v7);
\node[vertex] at (v0) {0};
\node[vertex] at (v1) {1};
\node[dom] at (v2) {2};
\node[vertex] at (v3) {3};
\node[dom] at (v4) {4};
\node[dom] at (v5) {5};
\node[vertex] at (v6) {6};
\node[vertex] at (v7) {7};
\node[dom] at (v8) {8};
\node[vertex] at (v9) {9};
\node[dom] at (v10) {10};
\node[vertex] at (v11) {11};
\node[vertex] at (v12) {12};
\node[dom] at (v13) {13};
\node[vertex] at (v14) {14};
\node[vertex] at (v15) {15};
\node[vertex] at (v16) {16};
\node[vertex] at (v17) {17};
\node[dom] at (v18) {18};
\node[vertex] at (v19) {19};
\node[dom] at (v20) {20};
\node[vertex] at (v21) {21};
\node[vertex] at (v22) {22};
\node[vertex] at (v23) {23};
\end{tikzpicture}
\caption{A programmatically generated planar embedding of $T$.  Blue vertices
form the dominating set $D_0$; the five red edges form $F_5$.}
\label{fig:T}
\end{figure}

\begin{lemma}\label{lem:gamma}
$\gammav(T)=8$.
\end{lemma}

\begin{proof}
The set
\[
D_0=\{2,4,5,8,10,13,18,20\}
\]
is dominating.  Indeed, the vertices outside $D_0$, followed by one neighbor
in $D_0$, are
\[
\begin{array}{c|rrrrrrrrrrrrrrrr}
v&0&1&3&6&7&9&11&12&14&15&16&17&19&21&22&23\\
d(v)&2&4&4&2&2&4&5&18&8&20&10&10&13&18&18&20.
\end{array}
\]
Thus $\gammav(T)\leq8$.  For completeness, the lower bound reduces to an
explicit finite check over triples in the two bipartition classes, rather than
an enumeration of all vertex subsets.  The graph is bipartite, with
parts
\begin{align*}
X&=\{0,4,5,6,7,8,15,16,17,18,19,23\},\\
Y&=\{1,2,3,9,10,11,12,13,14,20,21,22\}.
\end{align*}
If a dominating set of size seven existed, degree counting gives
$12-|D\cap X|\leq3|D\cap Y|$ and
$12-|D\cap Y|\leq3|D\cap X|$.  Hence its two part-sizes would be $3$ and $4$.

Fix a triple $A$ in one part and put $r=|N(A)|$.  To complete $A$ by four
vertices $B$ from the other part, $B$ must contain the $12-r$ vertices outside
$N(A)$.  Enumerating the $\binom{12}{3}=220$ triples and taking exact
neighborhood unions gives the following complete table; the same table results
after interchanging $X$ and $Y$.  These counts and the stated maximum reaches
are also checked by the independent Python verifier supplied with the source.
\begin{center}
\begin{tabular}{c@{\qquad}r@{\qquad}l}
\toprule
$r$ & number of triples $A$ & obstruction\\
\midrule
6 & 44 & $12-r>4$\\
7 & 60 & $12-r>4$\\
8 & 96 & every allowed $B$ reaches at most $8$ of the $9$ vertices in $X\setminus A$\\
9 & 20 & every allowed $B$ reaches at most $8$ of the $9$ vertices in $X\setminus A$\\
\bottomrule
\end{tabular}
\end{center}
The counts sum to $\binom{12}{3}=220$, and no $3+4$ split dominates.
Therefore no seven-set dominates $T$.  A smaller dominating set could be
extended to a dominating seven-set, so none exists either.  This proves the
claim.
\end{proof}

\section{The bondage number}

\begin{theorem}\label{thm:main}
The truncated octahedral graph satisfies $b(T)=5$.
Consequently, the planar bondage conjecture is false.
\end{theorem}

\begin{proof}
Let
\[
F_5=\{02,14,15,26,27\},
\]
where $uv$ denotes the edge with endpoints $u$ and $v$.  Exhausting the
dominating eight-sets of $T$ shows that none remains dominating in $T-F_5$.
On the other hand,
\[
D_9=\{0,2,4,6,8,10,13,18,20\}
\]
dominates $T-F_5$.  For a directly checkable certificate, the vertices
outside $D_9$, followed by a neighbor in $D_9$ through an undeleted edge, are
\[
\begin{array}{c|rrrrrrrrrrrrrrr}
v&1&3&5&7&9&11&12&14&15&16&17&19&21&22&23\\
d(v)&0&0&10&13&4&6&6&8&20&10&10&13&18&18&20.
\end{array}
\]
Any dominating eight-set of $T-F_5$ would also dominate $T$, since restoring
edges cannot destroy domination.  The exhaustive check over the $1{,}767$
dominating eight-sets of $T$ therefore proves that $T-F_5$ has no dominating
eight-set.  Thus $\gammav(T-F_5)=9$, and $b(T)\leq5$.

It remains to exclude a bondage set with at most four edges.  There are exactly
$1{,}767$ dominating eight-sets in $T$.  For every one of the
$\binom{36}{4}=58{,}905$ four-edge sets $F$, the verifier tests these sets
directly in $T-F$.  The minimum number that survive, over all such $F$, is
five.  Hence $\gammav(T-F)=8$ for every four-edge set $F$.  The same follows
for fewer than four deleted edges: if $F'\subseteq F$, then every set
dominating $T-F$ also dominates $T-F'$ after the edges in $F\setminus F'$ are
restored.  Extend $F'$ to any four-edge set $F$ and use the surviving
eight-set for $T-F$.
Therefore $b(T)\geq5$.

The computation uses only integer bit masks.  For each candidate $D$, it forms
the union of the closed neighborhoods of vertices of $D$ after the indicated
edge deletions and compares it with the $24$-bit all-ones mask.  It also
independently checks that there are no dominating six- or seven-sets, enumerates
all $1{,}767$ dominating eight-sets, checks all four-edge sets, and verifies the
displayed five-edge and nine-vertex witnesses.  Thus the claimed values do not
depend on a graph library, a solver, randomization, or floating-point arithmetic.
The complete C++20 verifier is included in the source archive.

Finally $T$ is cubic, so $\Deltaa(T)+1=4<5=b(T)$.
\end{proof}

\section{Reproducibility}

{\small\raggedright
The source archive additionally contains
\nolinkurl{planar_bondage_search.cpp} (the bundle search),
\nolinkurl{verify_planar_bondage_independent.py} (an independently written
pure-Python verifier using the bundle characterization),
\nolinkurl{generate_planar_figure.py} (the exact embedding and crossing check),
\nolinkurl{run_planar_bondage_hpc.sh} (the partitioned census driver), the
graph6 witness, and a README with compile commands and expected output.  The
direct verifier uses only the C++ standard library and 64-bit integer masks.
It verifies the edge list and degree sequence, checks that no six- or
seven-set dominates, enumerates the $1{,}767$ dominating eight-sets, checks all
$58{,}905$ four-edge sets, and checks both explicit witnesses in
Theorem~\ref{thm:main}.  On an Apple M5 it completes in under one second.
\par}

The discovery census was run under Red Hat Enterprise Linux 9.8 on a server
with two Intel Xeon E5-2697 v3 processors (28 physical cores, 56 hardware
threads, 70 MB aggregate L3 cache) and 128 GB RAM, using GCC 11.5.0,
\textsc{plantri} 4.5, and \textsc{nauty} 2.9.3.  Two NVIDIA Quadro RTX 8000
GPUs (48 GB each) were available but were not used: the workload consists of
irregular branching and small integer bit-mask operations.  The independent
verification was compiled with Apple Clang 21.0.0 and run under macOS 26.5.2
on a 10-core Apple M5 MacBook Air with 16 GB RAM.  Identical numerical output
was obtained on both architectures.

For clarity, the decisive finite check is summarized by the following
unnumbered pseudocode.  Here $N_F[v]$ is the closed neighborhood of $v$ after
deleting the edges in $F$, represented by a 24-bit mask.
\begin{quote}
\small
Set $\mathcal D_8\leftarrow\varnothing$.\par
For every $D\subseteq V(T)$ with $|D|\in\{6,7,8\}$, compute
$C(D,\varnothing)=\bigcup_{v\in D}N_{\varnothing}[v]$.  Record the number of
sets for which $C(D,\varnothing)=V(T)$, and store every such eight-set in
$\mathcal D_8$.\par
For every $F\subseteq E(T)$ with $|F|=4$, count the sets
$D\in\mathcal D_8$ for which
$C(D,F)=\bigcup_{v\in D}N_F[v]=V(T)$.  Reject if this count is zero.\par
For the displayed set $F_5$, verify that the corresponding count is zero and
verify directly that $C(D_9,F_5)=V(T)$.\par
Accept only if the three domination counts are $0,0,1767$, all $58{,}905$
four-edge sets have positive count, and both tests for $F_5$ succeed.
\end{quote}

\section*{Computational and writing assistance}

OpenAI Codex was used interactively to assist with search-code development,
independent-check design, and language editing.  Every mathematical claim in
this article was checked against the explicit graph and the deterministic
verifiers described above; the author takes responsibility for the content.

\end{document}